\documentclass[12pt,reqno]{amsart}

\usepackage[square,numbers]{natbib}              

\usepackage[centertags]{amsmath}        
\usepackage{amsfonts}
\usepackage{amssymb}
\usepackage{amsthm}
\usepackage{psfrag}                     
\usepackage{dcolumn}                    
\usepackage{graphicx}                   
\usepackage{color}                      
\usepackage[small]{caption}             
\usepackage[it]{subfig}                 
\usepackage{enumerate}                  
\usepackage{mathrsfs}                   
\usepackage{longtable}                  
\usepackage{lscape}                     
\usepackage{afterpage}                  
\usepackage[figuresright]{rotating}     

\usepackage[all,cmtip]{xy}
\usepackage{comment}
\usepackage{mathdots}

\usepackage{mathrsfs}
\usepackage[all]{xy}
\DeclareMathAlphabet{\mathpzc}{OT1}{pzc}{m}{it}

\usepackage{tikz}
\usetikzlibrary{arrows,decorations.pathmorphing,decorations.pathreplacing,positioning,shapes.geometric,shapes.misc,decorations.markings,decorations.fractals,calc,patterns}

\usepackage{graphicx}

\entrymodifiers={+!!<0pt,\fontdimen22\textfont2>}

\newcommand\xqed[1]{%
  \leavevmode\unskip\penalty9999 \hbox{}\nobreak\hfill
  \quad\hbox{#1}}
\newcommand\demo{\xqed{$\diamond$}}

\def\cA{\mathscr{A}}

\def\cC{\mathscr{C}}

\def\cR{\mathscr{R}}
\def\cS{\mathscr{S}}

\def\BC{\mathbb{C}}

\def\BN{\mathbb{N}}

\def\BZ{\mathbb{Z}}

\def\fD{\mathfrak{D}}

\def\fP{\mathfrak{P}}

\def\fa{\mathfrak{a}}
\def\fb{\mathfrak{b}}

\def\sC{\mathsf{C}}
\def\sD{\mathsf{D}}

\def\add{\operatorname{add}}

\def\adots{\mathinner{\mkern1mu\raise1.0pt\vbox{\kern7.0pt\hbox{.}}\mkern2mu\raise4.0pt\hbox{.}\mkern2mu\raise7.0pt\hbox{.}\mkern1mu}}

\def\ast{{\textstyle *}}

\def\Gr{\operatorname{Gr}}
\def\dim{\operatorname{dim}}

\def\dddots{\mathinner{\mkern1mu\raise10.0pt\vbox{\kern7.0pt\hbox{.}}\mkern2mu\raise5.3pt\hbox{.}\mkern2mu\raise1.0pt\hbox{.}\mkern1mu}}
\def\dddotssmall{\mathinner{\mkern1mu\raise7.0pt\vbox{\kern7.0pt\hbox{.}}\mkern-1mu\raise4pt\hbox{.}\mkern-1mu\raise1.0pt\hbox{.}\mkern1mu}}

\def\dim{\operatorname{dim}}

\def\End{\operatorname{End}}

\def\Ext{\operatorname{Ext}}

\def\Gr{\mathsf{Gr}}

\def\Hom{\operatorname{Hom}}

\def\Image{\operatorname{Im}}

\def\ind{\operatorname{ind}}

\def\Ker{\operatorname{Ker}}

\def\mod{\mathsf{mod}}
\def\Mod{\mathsf{Mod}}

\def\obj{\operatorname{obj}}

\def\SL2{\operatorname{SL}_2}

\def\Vect{\mathsf{Vect}}

\newtheorem{Lemma}{Lemma}[section]
\newtheorem{Theorem}[Lemma]{Theorem}
\newtheorem{Proposition}[Lemma]{Proposition}
\newtheorem{Corollary}[Lemma]{Corollary}

\theoremstyle{definition}
\newtheorem{Definition}[Lemma]{Definition}
\newtheorem{Setup}[Lemma]{Setup}

\newtheorem{Remark}[Lemma]{Remark}

\newtheorem{Example}[Lemma]{Example}

\DeclareMathSymbol{\Gamma}{\mathalpha}{letters}{"00}
\DeclareMathSymbol{\Delta}{\mathalpha}{letters}{"01}
\DeclareMathSymbol{\Theta}{\mathalpha}{letters}{"02}
\DeclareMathSymbol{\Lambda}{\mathalpha}{letters}{"03}
\DeclareMathSymbol{\Xi}{\mathalpha}{letters}{"04}
\DeclareMathSymbol{\Pi}{\mathalpha}{letters}{"05}
\DeclareMathSymbol{\Sigma}{\mathalpha}{letters}{"06}
\DeclareMathSymbol{\Upsilon}{\mathalpha}{letters}{"07}
\DeclareMathSymbol{\Phi}{\mathalpha}{letters}{"08}
\DeclareMathSymbol{\Psi}{\mathalpha}{letters}{"09}
\DeclareMathSymbol{\Omega}{\mathalpha}{letters}{"0A}

\theoremstyle{definition}

\theoremstyle{definition}

\theoremstyle{definition}

\newsavebox{\astrutbox}
\sbox{\astrutbox}{\rule[-5pt]{0pt}{20pt}}

\begin{document}

\setlength{\parindent}{0pt}
\setlength{\parskip}{6pt}

\title[A Caldero-Chapoton map depending on a torsion class]
{A Caldero-Chapoton map depending on a torsion class}

\author{Thomas A. Fisher}
\address{School of Mathematics and Statistics,
Newcastle University, Newcastle upon Tyne, NE1 7RU,
United Kingdom}
\email{t.a.fisher@ncl.ac.uk}





\thanks{2010 {\em Mathematics Subject Classification. }05E10, 13F60, 18E30}

\thanks{{\em Key words and phrases. }Cluster category, polygon dissection, Ptolemy diagram, torsion class, triangulated category}


\maketitle

\begin{abstract}
Frieze patterns of integers were studied by Conway and Coxeter, see \citep{CCI} and \citep{CCII}. Let $\cC$ be the cluster category of Dynkin type $A_n$. Indecomposables in $\cC$ correspond to diagonals in an $(n+3)$-gon. Work done by Caldero and Chapoton showed that the Caldero-Chapoton map (which is a map dependent on a fixed object $R$ of a category, and which goes from the set of objects of that category to $\BZ$), when applied to the objects of $\cC$ can recover these friezes, see \citep{CaCh}. This happens precisely when $R$ corresponds to a triangulation of the $(n+3)$-gon, i.e. when $R$ is basic and cluster tilting.  Later work (see \citep{BHJ}, \citep{HJROI}) generalised this connection with friezes further, now to $d$-angulations of the $(n+3)$-gon with $R$ basic and rigid. In this paper, we extend these generalisations further still, to the case where the object $R$ corresponds to a general Ptolemy diagram, i.e. $R$ is basic and $\add(R)$ is the most general possible torsion class (where the previous efforts have focused on special cases of torsion classes).
\end{abstract}

\setcounter{section}{-1}
\section{Introduction}
\label{sec:introduction}

In \citep{Cox}, \citep{CCI} and its direct sequel \citep{CCII}, frieze patterns of integers were considered by Conway and Coxeter. A frieze pattern of Dynkin type $A_n$ consists of $n+2$ infinite, interlacing rows of positive integers (with both the top and bottom row consisting entirely of ones) satisfying the so-called {\textit{unimodular rule}}, which says for all adjacent numbers forming a diamond in the frieze
\[ 
  \xymatrix @-2.0pc @! {
  & b & \\
  a & & d \\
  & c & \\
               }
\]
we have $ad-bc=1$. An example frieze lifted from \citep{CCI} is illustrated below.
\[ \def\objectstyle{\scriptstyle}
  \xymatrix @-1.5pc @! {
  & 1 & & 1 & & 1 & & 1 & & 1 & & 1 & & \cdots  \\
  \cdots & & 3 & & 1 & & 2 & & 2 & & 1 & & 3 & \\
  & 2 & & 2 & & 1 & & 3 & & 1 & & 2 & & \cdots \\
  \cdots & & 1 & & 1 & & 1 & & 1 & & 1 & & 1 &
               }
\] \\
Let $\cC$ be a category which is $\BC$-linear and Hom-finite, and let $c, R \in \obj\cC$. In \citep[Section 5]{CaCh}, Caldero and Chapoton discovered that the formula
$$\rho(c)=\chi ( \Gr\,Gc)$$
in a very special case recovers Conway-Coxeter friezes. Here,
$$G(-) = \Hom_\cC(R,-) : \cC \rightarrow \mod(B)$$
is a functor, where $B = \End_\cC(R)$ is the endomorphism algebra and $\mod$ is the category of finitely generated right modules, $\Gr\,Gc $ is the {\textit{module Grassmannian of submodules}} of $Gc$, and $\chi$ (defined by cohomology with compact support) is the {\textit{Euler characteristic}}, see \citep[Page 93]{Fult}. In particular, Caldero and Chapoton let the category $\cC$ be the cluster category $\sC(A_n)$ of Dynkin type $A_n$, see \citep[Section 1]{BMRRT} and \citep[Section 5]{CCS}. Now, indecomposables in $\sC(A_n)$ correspond to diagonals in an $(n+3)$-gon. Caldero and Chapoton let the object $R$ represent a triangulation; this is equivalent to $\add(R)$ being a cluster tilting subcategory of $\sC(A_n).$ The formula then recovered the combinatorial work of Conway and Coxeter, yielding friezes of type $A_n$, see \citep[Section 5]{CaCh}. 

In \citep{BHJ}, Bessenrodt, Holm and J{\o}rgensen generalised the work done by Conway and Coxeter \citep{CCI}, \citep{CCII} to $d$-angulations of polygons in a purely combinatorial way. In a later paper \citep{HJROI}, Holm and J{\o}rgensen proved that the Caldero-Chapoton formula gives the generalised friezes in \citep{BHJ} where now $R$ corresponds to a $d$-angulation, not a triangulation. See \citep{HJROII} for a generalisation to polynomial values. We remark that in both \citep{BHJ} and \citep{HJROI}, everything works for general polygon dissections, not just $d$-angulations, and that a polygon dissection corresponds to $R$ being a rigid object.

Now note that in the cases considered so far, $\add(R)$ is a torsion class in $\sC(A_n)$ (i.e. it is precovering and closed under extensions), albeit a rather special one. It is natural to consider the ``ultimate generalisation'' where $\add(R)$ is a general torsion class, and that is what we shall do in this paper. By \citep{HJR}, this means that $R$ corresponds to a so-called {\textit{Ptolemy diagram}}.

\begin{Definition}
Let $P = P_n$ be an $n$-gon with vertices labelled $1,2,\dots,n$. Let $\fD = \fD(P_n)$ be a set of diagonals of $P_n$. A Ptolemy diagram is an $n$-gon $P$ together with a set of diagonals which satisfy the following rule. Let $\{a_1,a_2\}$ and $\{b_1,b_2\}$ be two diagonals in $\fD$. If $\{a_1,a_2\}$ and $\{b_1,b_2\}$ cross, then all pairs $\{a_1,b_1\}$, $\{a_1,b_2\}$, $\{a_2,b_1\}$ and $\{a_2,b_2\}$ which are diagonals are also in $\fD$. \demo
\end{Definition}

\begin{Example}
 The following is an example of a Ptolemy diagram with eight vertices.  \\
\[
  \begin{tikzpicture}[auto]
    \node[name=s, shape=regular polygon, regular polygon sides=8, minimum size=5.5cm, draw] {}; 
    \draw[thick] (s.corner 2) to (s.corner 4);
    \draw[thick] (s.corner 1) to (s.corner 5);
    \draw[thick] (s.corner 5) to (s.corner 8);
    \draw[thick] (s.corner 2) to (s.corner 5);
    \draw[thick] (s.corner 1) to (s.corner 4);
        \draw[shift=(s.corner 1)] node[above] {$1$};
    \draw[shift=(s.corner 2)] node[above] {$2$};
    \draw[shift=(s.corner 3)] node[left] {$3$};
    \draw[shift=(s.corner 4)] node[left] {$4$};
    \draw[shift=(s.corner 5)] node[below] {$5$};
    \draw[shift=(s.corner 6)] node[below] {$6$};
    \draw[shift=(s.corner 7)] node[right] {$7$};
    \draw[shift=(s.corner 8)] node[right] {$8$};
  \end{tikzpicture} 
\] \\
Here, the diagonals can be given by the pairs of integers $\{2,4\},\{2,5\},\{1,4\},\{1,5\},$ and $\{5,8\}.$ \demo
\end{Example}

We also remind the reader of the definition of a {\textit{rigid object}}.

\begin{Definition}
Let $\cC$ be a $\BC$-linear triangulated category with suspension functor $\Sigma$. An object $R$ of $\cC$ is said to be rigid if it satisfies $\Hom_\cC(R,\Sigma R)=0$. \demo
\end{Definition}

%

\section{Background material}
\label{sec:Backgroundmat}

In this section, let $\cC = \sC(A_n)$, the cluster category of type $A_n$, where $n \geq 3$ is an integer.

\begin{Remark} \label{rem:cluscatintro}
It is well known (see \citep[Cor. 4.5(i)]{HAPPEL1}, \citep[p. 359 Table]{MYDPG}, \citep[Prop. 1.3]{BMRRT}) that the Auslander-Reiten quiver of $\cC$ is $\BZ A_n$ modulo a glide reflection, illustrated below with the example of $\sC(A_7)$.
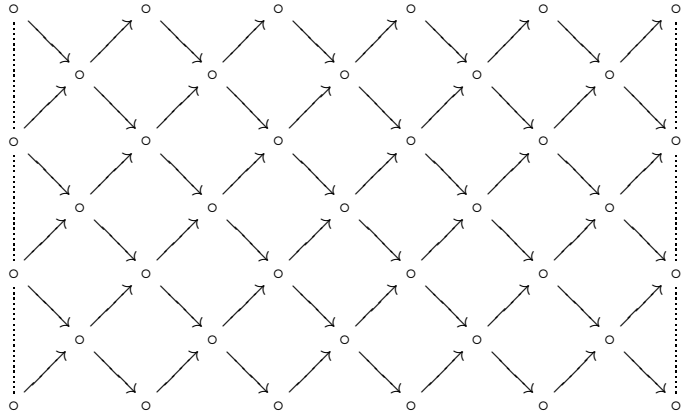
\[ \def\objectstyle{\scriptstyle}
  \xymatrix @-0.85pc @! {
     \circ \ar@{.}[dd] \ar[dr] & & \circ \ar[dr] & & \circ \ar[dr] & & \circ \ar[dr] & & \circ \ar[dr] & & \circ \ar@{.}[dd] \\
     & \circ \ar[ur] \ar[dr] & & \circ \ar[ur] \ar[dr] & & \circ \ar[ur] \ar[dr] & & \circ \ar[ur] \ar[dr] & & \circ \ar[ur] \ar[dr] &  \\
     \circ \ar@{.}[dd] \ar[dr] \ar[ur] & & \circ \ar[dr] \ar[ur] & & \circ \ar[dr] \ar[ur] & & \circ \ar[dr] \ar[ur] & & \circ \ar[dr] \ar[ur] & & \circ \ar@{.}[dd] \\
     & \circ \ar[ur] \ar[dr] & & \circ \ar[ur] \ar[dr] & & \circ \ar[ur] \ar[dr] & & \circ \ar[ur] \ar[dr] & & \circ \ar[ur] \ar[dr] & \\
     \circ \ar@{.}[dd] \ar[ur] \ar[dr] & & \circ \ar[ur] \ar[dr] & & \circ \ar[ur] \ar[dr] & & \circ \ar[ur] \ar[dr] & & \circ \ar[ur] \ar[dr] & & \circ \ar@{.}[dd] \\
     & \circ \ar[ur] \ar[dr] & & \circ \ar[ur] \ar[dr] & & \circ \ar[ur] \ar[dr] & & \circ \ar[ur] \ar[dr] & & \circ \ar[ur] \ar[dr] & \\
     \circ \ar[ur] & & \circ \ar[ur] & & \circ \ar[ur] & & \circ \ar[ur] & & \circ \ar[ur] & & \circ \\
               }
\]

It has a coordinate system on it, detailed below.
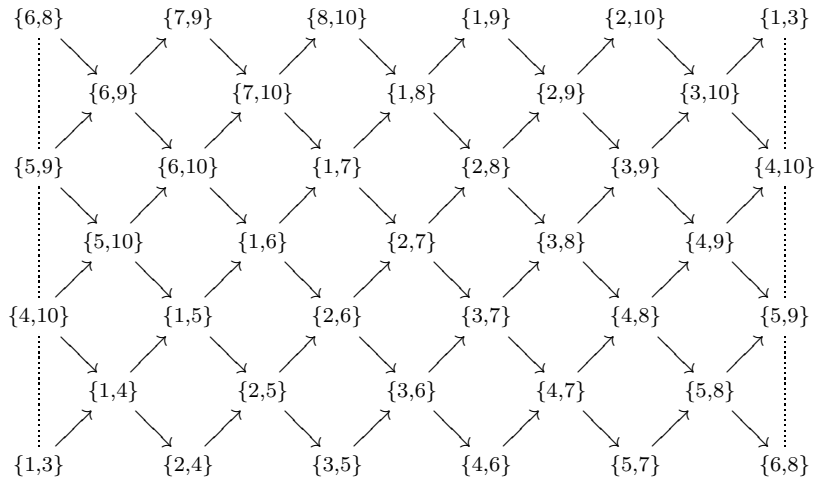
\[ \def\objectstyle{\scriptstyle}
  \xymatrix @-2.2pc @! {
     \{6,8\} \ar@{.}[dd] \ar[dr] & & \{7,9\} \ar[dr] & & \{8,10\} \ar[dr] & & \{1,9\} \ar[dr] & & \{2,10\} \ar[dr] & & \{1,3\} \ar@{.}[dd] \\
     & \{6,9\} \ar[ur] \ar[dr] & & \{7,10\} \ar[ur] \ar[dr] & & \{1,8\} \ar[ur] \ar[dr] & & \{2,9\} \ar[ur] \ar[dr] & & \{3,10\} \ar[ur] \ar[dr] & \\
     \{5,9\} \ar@{.}[dd] \ar[dr] \ar[ur] & & \{6,10\} \ar[dr] \ar[ur] & & \{1,7\} \ar[dr] \ar[ur] & & \{2,8\} \ar[dr] \ar[ur] & & \{3,9\} \ar[dr] \ar[ur] & & \{4,10\} \ar@{.}[dd] \\
     & \{5,10\} \ar[ur] \ar[dr] & & \{1,6\} \ar[ur] \ar[dr] & & \{2,7\} \ar[ur] \ar[dr] & & \{3,8\} \ar[ur] \ar[dr] & & \{4,9\} \ar[ur] \ar[dr] & \\
     \{4,10\} \ar@{.}[dd] \ar[ur] \ar[dr] & & \{1,5\} \ar[ur] \ar[dr] & & \{2,6\} \ar[ur] \ar[dr] & & \{3,7\} \ar[ur] \ar[dr] & & \{4,8\} \ar[ur] \ar[dr] & & \{5,9\} \ar@{.}[dd] \\
     & \{1,4\} \ar[ur] \ar[dr] & & \{2,5\} \ar[ur] \ar[dr] & & \{3,6\} \ar[ur] \ar[dr] & & \{4,7\} \ar[ur] \ar[dr] & & \{5,8\} \ar[ur] \ar[dr] & \\
     \{1,3\} \ar[ur] & & \{2,4\} \ar[ur] & & \{3,5\} \ar[ur] & & \{4,6\} \ar[ur] & & \{5,7\} \ar[ur] & & \{6,8\} \\
               }
\]

In general, the coordinates are taken modulo $n+3$ and we have the following picture.
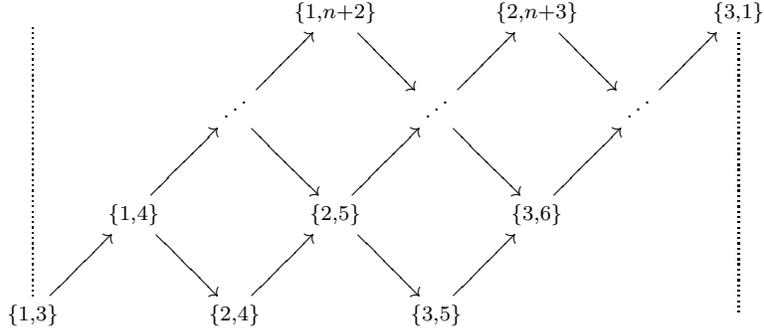
\[ \def\objectstyle{\scriptstyle}
  \xymatrix @-2.0pc @! {
     & & & \{1,n+2\} \ar[dr] & & \{2,n+3\} \ar[dr] & & \{3,1\} \ar@{.}[ddd] \\
     & & \iddots \ar[ur] \ar[dr] & & \iddots \ar[ur] \ar[dr] & & \iddots \ar[ur] &  \\
     & \{1,4\} \ar[ur] \ar[dr] & & \{2,5\} \ar[ur] \ar[dr] & & \{3,6\} \ar[ur] & &  \\
     \{1,3\} \ar@{.}[uuu] \ar[ur] & & \{2,4\} \ar[ur] & & \{3,5\} \ar[ur] & &  &  \\
               }
\]

The coordinate system used above in the Auslander-Reiten quiver has a very nice property. Namely, one can identify a coordinate pair with a diagonal in the $(n+3)$-gon such that if $a$ and $b$ are indecomposable objects with corresponding diagonals $\fa$ and $\fb$, then
$$\dim_{ \BC } \Ext^1_{ \cC }( a,b )
  = \left\{
      \begin{array}{cl}
        1 & \mbox{ if $\fa$ and $\fb$ cross, } \\[1mm]
        0 & \mbox{ if not. }
      \end{array}
    \right.
$$
See, e.g. \citep[Lemma 2.1, Thm. 4.4]{CCS}. \demo
\end{Remark}

\section{Properties of the Caldero-Chapoton Map}
\label{sec:CCProperties}

In this section, the requirement that $\cC = \sC(A_n)$ is relaxed, although, in particular, the results in this section apply to $\sC(A_n)$.

\begin{Setup} \label{set:rhoonadd}
Let $\cC$ be a $\BC$-linear, Hom-finite, triangulated category with an object $R$, and set $B=\End_\cC(R)$. Define the functor 
$$G(-) = \Hom_\cC(R,\Sigma(-)) : \cC \rightarrow \mod(B)$$
and define the (modified) Caldero-Chapoton map
$$\rho (c) = \chi ( \Gr\,Gc)$$
on the objects $c$ of $\cC$. Here, $\Gr\,Gc $ is the {\textit{module Grassmannian of submodules}} of $Gc$, and $\chi$ (defined by cohomology with compact support) is the {\textit{Euler characteristic}}, see \citep[Page 93]{Fult}. \demo
\end{Setup}

\begin{Lemma} \label{lem:rhoismult}
In the situation of Setup \ref{set:rhoonadd}, we have for $x,y \in \operatorname{obj} \cC,$ that $ \rho(x \oplus y) = \rho(x)\rho(y),$ see \citep[Corollary 3.7]{CaCh}.
\end{Lemma}

\begin{proof}
Consider the sequence of morphisms
$$
\xymatrix{
x \ar[rr]^-{\begin{pmatrix} 1 \\ 0  \end{pmatrix}} && x\oplus y \ar[rr]^-{\begin{pmatrix} 0 & 1  \end{pmatrix}} && y
}
$$
and apply $G$ to obtain
$$
\xymatrix{
Gx \ar[rr]^-{\iota} && Gx\oplus Gy \ar[rr]^-{\pi} && Gy
}
$$
where $\iota$ and $\pi$ are the natural injection and projection respectively. Apply $\Gr$ to obtain
$$
\xymatrix{
\Gr\,Gx && \ar[ll]^{} \Gr\,G(x\oplus y) \ar[rr]^{} && \Gr\,Gy
}
$$
$$
\xymatrix{
\iota^{-1}M && \ar@{|->}[ll]^{} M \ar@{|->}[rr]^{} && \pi M
}
$$
where $M \in \Gr\,G(x\oplus y)$. This yields $\Gamma$, a surjective constructible map with affine fibres, defined by
$$\Gamma : \Gr\,G(x \oplus y) \rightarrow (\Gr\,Gx) \times (\Gr\,Gy), M \mapsto (\iota^{-1}M, \pi M),$$
see \citep[Lemma 3.11]{CaCh} which is stated for Auslander-Reiten sequences but has a proof which works in our situation. Taking $\chi$ of this gives the result, namely that 
$$\chi(\Gr\,G(x \oplus y)) = \chi( (\Gr\,Gx) \times (\Gr\,Gy)) = \chi(\Gr\,Gx) \chi (\Gr\,Gy)$$
$$\iff \rho(x \oplus y) = \rho(x)\rho(y),$$
see \citep[p. 93, Exercise]{Fult}.
\end{proof}

\begin{Lemma} \label{lem:splitepi}
If
$$\xymatrix{ \tau r \ar[r] & y \ar[r]^-{\upsilon} & r \ar[r]^-{\sigma} & \Sigma\tau r }$$
is an Auslander-Reiten triangle and
$$\xymatrix{ m \ar[r] & z \ar[r]^-{\zeta} & r \ar[r]^-{\delta} & \Sigma m }$$
is a distinguished triangle with $\delta \neq 0$, then $\exists h : \Sigma m \rightarrow \Sigma\tau r$ such that $\sigma = h \circ \delta$. 
\end{Lemma}

\begin{proof}
Because $\delta \neq 0$, $\zeta$ is not a split epimorphism, and because 
$$\xymatrix{ \tau r \ar[r] & y \ar[r]^{\upsilon} & r \ar[r]^{\sigma} & \Sigma\tau r }$$
is an Auslander-Reiten triangle, there exists $f : z \rightarrow y$ such that $\upsilon \circ f = \zeta$. The following diagram illustrates this.
$$
\xymatrix{
& & m \ar[d] &  \\
& & z \ar[ld]_-{f} \ar[d]_-{\zeta} \ar[dr]^-{0} &  \\
\tau r \ar[r] & y \ar[r]_-{\upsilon} & r \ar[r]_-{\sigma} \ar[d]_-{\delta} & \Sigma\tau r \\
 &  & \Sigma m & 
}
$$
Now, because $\sigma \circ \zeta = \sigma \circ \upsilon \circ f = 0 \circ f = 0,$ there exists $h : \Sigma m \rightarrow \Sigma\tau r$ such that $\sigma = h \circ \delta,$ as shown below.
$$
\xymatrix{
& & m \ar[d] &  \\
& & z \ar[ld]_-{f} \ar[d]_-{\zeta} \ar[dr]^-{0} &  \\
\tau r \ar[r] & y \ar[r]_-{\upsilon} & r \ar[r]_-{\sigma} \ar[d]_-{\delta} & \Sigma\tau r \\
 &  & \Sigma m \ar[ur]_-{h} & 
}
$$
\end{proof}

\begin{Setup} \label{set:mandr}
Let $\cC$ be a triangulated, $\BC$-linear, 2-Calabi-Yau category. Let $R \in \cC$ be an object, and let $B=\End_{\cC}(R)$.

Moreover, let $m,r$ be indecomposable objects of $\cC$ satisfying 
$$\dim_\BC\Ext^1(r,m) = \dim_\BC\Ext^1(m,r) = 1$$
and $\Ext^1(R,r)=0$. \demo
\end{Setup}

\begin{Setup} \label{set:mandr2}
We have the following distinguished triangles
$$\xymatrix{
m \ar[r]^-{\mu} & a \ar[r] & r \ar[r]^-{\delta} & \Sigma m,
}
$$
$$\xymatrix{
r \ar[r]^{} & b \ar[r]^-{\beta} & m \ar[r]^-{\epsilon} & \Sigma r
}
$$
where $\delta$, $\epsilon$ $\neq 0$ and due to the one-dimensionality of the $\Ext^1$-spaces, these are the unique non-split extensions between $m$ and $r$ up to isomorphism.

Upon applying $G$ to the distinguished triangles and rolling, we obtain the following exact sequences in $\mod(B)$:
$$\xymatrix{
G(\Sigma^{-1}r) \ar[rr]^-{G(\Sigma^{-1}\delta)} && Gm \ar[r]^-{G\mu} & Ga \ar[r]^{} & 0,
}
$$
$$\xymatrix{
0 \ar[r]^{} & Gb \ar[r]^-{G\beta} & Gm \ar[r]^-{G\epsilon} & G(\Sigma r).
}
$$
The zeroes result from the condition $\Ext^1(R,r)=0$. \demo
\end{Setup}

\begin{Lemma} \label{lem:MsubmodGm}
Let $M$ be a submodule of $Gm$. Either $\Ker G\mu \subseteq M$ or $M \subseteq \Image G\beta$, but not both.
\end{Lemma}

\begin{proof}
``Either/or'': Suppose $M \not \subseteq \Ker G\epsilon = \Image G\beta.$ Then we must show that $\Image G (\Sigma^{-1} \delta) = \Ker G\mu \subseteq M$.

Now, because $M \not \subseteq \Ker G\epsilon,$ there exists $x \in M$ such that $(G\epsilon)(x) \neq 0$. Note, $(G\epsilon)(x) = (\Sigma \epsilon)\circ x$ because after applying $G$ to the morphism $\epsilon : m \rightarrow \Sigma r$ we obtain $\Hom(R,\Sigma \epsilon) : (R,\Sigma m) \rightarrow (R, \Sigma^2 r).$

So now consider 
$$x^{\ast} = \Hom(x,\Sigma^2r) : \Hom(\Sigma m, \Sigma^2r) \rightarrow \Hom(R,\Sigma^2 r), \Sigma \epsilon \mapsto (\Sigma \epsilon) \circ x \neq 0$$ 
and consider the functor $D(-) = \Hom_\BC(-,\BC)$ applied to this, shown as the upper horizontal morphism in the commutative square below,
$$
\xymatrix{
D \Hom(R, \Sigma^2 r) \ar[rr]^-{D(x^{\ast})} \ar[dd]^{\cong} && D \Hom(\Sigma m, \Sigma^2 r) \ar[dd]^{\cong} \\ \\
\Hom(r,R) \ar[rr]_-{x_{\ast} = \Hom(r,x)} && \Hom(r,\Sigma m),
}
$$
which exists by Serre duality.

We have that $x^{\ast} \neq 0 \Rightarrow D(x^{\ast}) \neq 0 \Rightarrow x_{\ast} \neq 0 \Rightarrow x_{\ast}$ is surjective because $\dim_\BC\Hom(r,\Sigma m)=1$.

Since $x_{\ast}$ is onto, if $\delta : r \rightarrow \Sigma m$ is given then there exists $z : r \rightarrow R$ such that
$$
\xymatrix{
&& r \ar[ddll]_-{z} \ar[dd]^-{\delta = x_{\ast}(z)} \\ \\ 
R \ar[rr]_-{x} && \Sigma m
}
$$
is commutative.

Now, apply $G$ to $\Sigma^{-1}\delta : \Sigma^{-1}r \rightarrow m$ to obtain $\delta_{\ast} = \Hom(R,\delta) : \Hom(R,r) \rightarrow \Hom(R,\Sigma m)$ and note that requiring $\Image G(\Sigma^{-1}\delta) \subseteq M$ is equivalent to requiring that $\Image \delta_{\ast} \subseteq M$. 

So now let $w \in \Hom(R,r)$ and consider $\delta_{\ast}(w) = \delta \circ w$. The following diagram
$$
\xymatrix{
&& R \ar@/^-1.25pc/[ddddll]_-{z \circ w} \ar[dd]^-{w} \ar@/^2.50pc/[dddd]^-{\delta_{\ast}(w)} \\ \\
&& r \ar[ddll]_-{z} \ar[dd]^-{\delta} \\ \\
R \ar[rr]_-{x} && \Sigma m
}
$$
where $z$ exists by the above, helps illustrate that $$\delta_{\ast}(w) = \delta \circ w = x \circ z \circ w = x \cdot (z \circ w)$$ 
which is in $M$, as required, because $x \in M$ and $z \circ w \in B.$

``Not both'': There is an Auslander-Reiten triangle $\xymatrix{ \Sigma r \ar[r] & y \ar[r] & r \ar[r]^-{\sigma} & \Sigma^2 r,}$ and because $\xymatrix{r \ar[r]^{\delta} & \Sigma m} \neq 0,$ we have that $\sigma$ factors as $\xymatrix{r \ar[r]^-{\delta} & \Sigma m \ar[r]^-{\Psi} & \Sigma^2 r}$ by Lemma \ref{lem:splitepi}. Since $\Psi\delta \neq 0,$ we have that $\Psi \neq 0,$ and since $\Psi \in \Ext^1(\Sigma m, \Sigma r) \cong \Ext^1(m,r),$ which is one-dimensional, $\Psi$ is a nonzero scalar multiple of $\Sigma\epsilon$. This means that $\Sigma(\epsilon)\delta \neq 0,$ which implies $G(\epsilon \Sigma^{-1} \delta) \neq 0$. 

Now suppose $\Image G(\Sigma^{-1} \delta) \subseteq M.$ Apply $G\epsilon$ to give $\Image G(\epsilon \Sigma^{-1}\delta) \subseteq (G\epsilon)M$. Since $\Image G(\epsilon\Sigma^{-1} \delta)$ is nonzero, $(G\epsilon)M \neq 0, $ that is $M \not\subseteq \Ker G\epsilon$, as required.

\end{proof}

\begin{Proposition} \label{pro:rhomab}
We have $\rho(m) = \rho(a) + \rho(b)$.
\end{Proposition}

\begin{proof}
Take the two exact sequences
$$\xymatrix{
0 \ar[r]^{} & Gb \ar[r]^-{G\beta} & Gm \ar[r]^-{G\epsilon} & G(\Sigma r),
}
$$
$$\xymatrix{
G(\Sigma^{-1}r) \ar[rr]^-{G(\Sigma^{-1}\delta)} && Gm \ar[r]^-{G\mu} & Ga \ar[r]^{} & 0
}
$$
from Setup \ref{set:mandr2} and apply $\Gr$ to obtain the following morphisms of algebraic varieties.
$$\xymatrix{
\Gr\,Gb \ar@{^{(}->}[r] & \Gr\,Gm & \Gr\,Ga \ar@{_{(}->}[l]
}
$$
$$\xymatrix{
B \ar@{{|}->}[r] & (G\beta)(B) & 
}
$$
$$\xymatrix{
 & (G\mu)^{-1}(A) & A \ar@{{|}->}[l]
}
$$
Lemma \ref{lem:MsubmodGm} implies that $\Gr\,Gm$ is the disjoint union of $\Gr\,Ga$ and $\Gr\,Gb$. Since $G\beta$ and $G\mu$ are linear, the morphisms 
$$\xymatrix{\Gr\,Gb\, \ar@{^{(}->}[r] & \Gr\,Gm &\textrm{and}& \Gr\,Gm & \, \Gr\,Ga \ar@{_{(}->}[l]}$$
are constructible, and so their images inside $\Gr\,Gm$ are constructible subsets. Therefore the Euler characteristic of the disjoint union is the sum of the Euler characteristics, by \citep[p. 92, item (3)]{Fult}.
\end{proof}
 
\section{Ptolemy Diagrams and an Inductive Procedure for computing $\rho$}
\label{sec:PtolDiags}

Now and throughout the rest of this paper, let $\cC = \sC(A_n)$ be the cluster category of Dynkin type $A_n$, and recall Remark \ref{rem:cluscatintro} which describes some of its important properties, including its connection to the $(n+3)$-gon. Let $R$ be a basic object of $\cC$ and let $\cR$ be the full subcategory of $\cC$ given by $\add(R)$. Recall the definition $B = \End_\cC(R).$ 

\begin{Setup}
Let $P$ be an $(n+3)$-gon. As explained in Section \ref{sec:Backgroundmat}, there is a bijection between the indecomposable objects of $\cC$ and the diagonals of $P.$ We will think of edges and diagonals in $P$ in combinatorial terms by identifying an edge or a diagonal with its pair of endpoints, although it is worth keeping in mind the geometric interpretation of diagonals as line segments inside a geometric polygon. In this vein, let $P$ have vertex set $V_P = \{1,2,\dots,n+3\}$ and edge set $E_P =\{\{a,b\} \hspace{1.5mm} | \hspace{1.5mm} a,b \in V_P, \, |b-a| \in \{1,n+2\}\}$. Let $D_P = \{\{a,b\} \hspace{1.5mm} | \hspace{1.5mm} a,b \in V_P, \, |b-a| \geq 2, \{1,n+3\} \not\in D_P\}$ be the set of diagonals of $P$ and $D_P'$ be a subset of $D_P$ consisting of diagonals which are pairwise non-crossing. If in a given context, $P$ is fixed, then whilst working we drop the subscript notation and write $V, E, D, D'$ respectively.

Suppose $P$ and $D'$ are fixed. The diagonals in $D'$ are then referred to as {\textit{dissecting diagonals.}} The diagonals in $D'$ partition the $(n+3)$-gon into subpolygons called {\textit{cells}}. Note, if $D'$ has been chosen such that it is maximal, then the corresponding diagonals triangulate the $(n+3)$-gon, and each cell is a triangle. Conversely, if $D'$ is the empty set, then the corresponding partition yields a single cell, namely $P$ itself. 

For a particular choice of $D'$, we assign the corresponding cells labels $C_1,C_2,\dots,C_m$. Then we assign each cell $C_i$ a subset $T_i$ of $D$ as follows. Either $T_i$ is empty, or it is the subset of $D$ containing all interior diagonals of $C_i$. A nonempty cell is called a {\textit{clique}} and the interior diagonals are referred to as {\textit{clique diagonals}}. Note that when $C_i$ is a triangle it is both empty and a clique. Then the collection 
$$\fP = \displaystyle E \, \bigcup \, D' \, \bigcup_{i=1}^m \, T_i$$
is called a {\textit{Ptolemy diagram}} (see \citep[Definition 2.1]{HJR}) and an example of one is illustrated below for the case $n=5$. \\
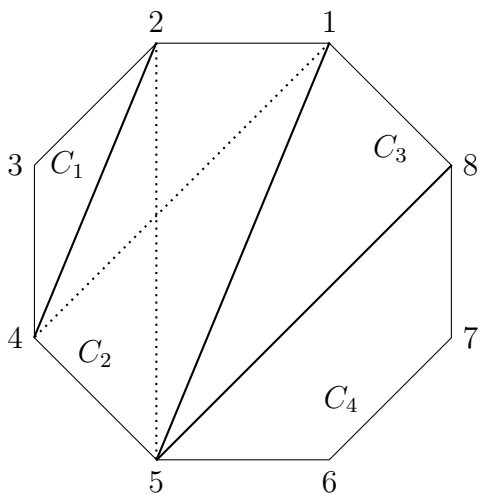
\[
  \begin{tikzpicture}[auto]
    \node[name=s, shape=regular polygon, regular polygon sides=8, minimum size=6cm, draw] {}; 
    \draw[thick] (s.corner 2) to (s.corner 4);
    \draw[shift=(s.corner 3)] node[right=2pt] {$C_1$};
    \draw[shift=(s.side 4)] node[above=8pt] {$C_2$};
    \draw[shift=(s.side 8)] node[below=8pt] {$C_3$};
    \draw[shift=(s.side 6)] node[left=8pt] {$C_4$};
    \draw[thick] (s.corner 1) to (s.corner 5);
    \draw[thick] (s.corner 5) to (s.corner 8);
    \draw[thick,dotted] (s.corner 2) to (s.corner 5);
    \draw[thick,dotted] (s.corner 1) to (s.corner 4);
        \draw[shift=(s.corner 1)] node[above] {$1$};
    \draw[shift=(s.corner 2)] node[above] {$2$};
    \draw[shift=(s.corner 3)] node[left] {$3$};
    \draw[shift=(s.corner 4)] node[left] {$4$};
    \draw[shift=(s.corner 5)] node[below] {$5$};
    \draw[shift=(s.corner 6)] node[below] {$6$};
    \draw[shift=(s.corner 7)] node[right] {$7$};
    \draw[shift=(s.corner 8)] node[right] {$8$};
  \end{tikzpicture} 
\] \\
Here, $P$ is an octagon and $D' = \{\{2,4\}, \{1,5\}, \{5,8\}\}$ is the set of dissecting diagonals, illustrated above with solid lines. We also have that $T_i=\emptyset$ for $i \in \{1,3,4\}$ and for $i=2,$ $T_i$ consists of two diagonals, $\{2,5\}$ and $\{1,4\}$, illustrated above with dotted lines.

Any given cell can be uniquely identified by a subset of $V$. In the example above, $C_1$ is given by $\{2,3,4\},$ $C_2$ is given by $\{1,2,4,5\},$ $C_3$ is given by $\{1,5,8\}$ and $C_4$ is given by $\{5,6,7,8\}.$ \demo
\end{Setup}

\begin{Remark} \label{rem:Ptolemysetup}
Let $\fP$ be a Ptolemy diagram. Take the collection of all the diagonals in $D'$ and $T_i$ for each $i$ and set the object $R \in \obj\cC$ to be the direct sum of the corresponding indecomposables. By \citep{HJROI}, the subcategory $\add(R)$ is a torsion class of $\cC,$ and each torsion class has this form. We henceforth use this $R$ in the Caldero-Chapoton map $\rho$ from Setup \ref{set:rhoonadd}.

Now let $r$ be any indecomposable summand of $R$. If $r$ corresponds to a dissecting diagonal, then because a dissecting diagonal crosses no other diagonal in $\fP,$ we have that $\Ext^1(R,r)=0.$ If $m \in \ind(\cC)$ is such that 
$$\dim_\BC\Ext^1(r,m) = \dim_\BC\Ext^1(m,r) = 1$$
then Proposition \ref{pro:rhomab} and \citep[Figure 5]{HJROI} together say that $\rho,$ the Caldero-Chapoton map from Setup \ref{set:rhoonadd}, satisfies $\rho(m)=\rho(a)+\rho(b)$ where $a = a' \oplus a''$ and $b = b' \oplus b''$ have diagonals corresponding to

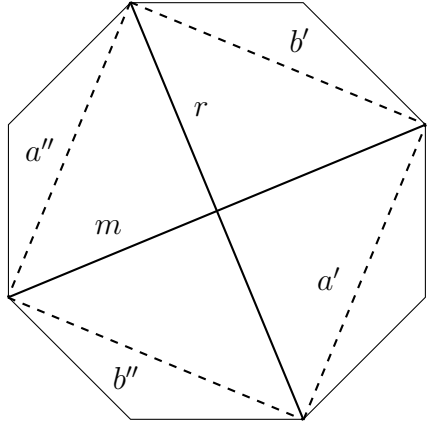
\[
\vspace{0.9mm}
  \begin{tikzpicture}[auto]
    \node[name=s, shape=regular polygon, regular polygon sides=8, minimum size=6cm, draw] {}; 
    \draw[thick] (s.corner 4) to node[pos=0.3] {$m$} (s.corner 8);
    \draw[thick,dashed] (s.corner 2) to node[pos=0.5,left=1pt] {$a''$} (s.corner 4);
    \draw[thick,dashed] (s.corner 2) to node[pos=0.5] {$b'$} (s.corner 8);
    \draw[thick,dashed] (s.corner 4) to node[pos=0.4,below=3pt] {$b''$} (s.corner 6);
    \draw[thick,dashed] (s.corner 6) to node[pos=0.4] {$a'$} (s.corner 8);
    \draw[thick] (s.corner 2) to node[pos=0.3] {$r$} (s.corner 6);
  \end{tikzpicture} 
  \vspace{0.9mm}
\]

where if any of $a', a'', b', b''$ are edges of the polygon, then they are interpreted as the zero object. Using Lemma \ref{lem:rhoismult}, this means $\rho(m)=\rho(a')\rho(a'')+\rho(b')\rho(b'').$ \demo
\end{Remark}

\begin{Remark}
Remark \ref{rem:Ptolemysetup} provides an inductive method to compute $\rho(m)$. Namely, because $\rho(m)=\rho(a)+\rho(b) = \rho(a' \oplus a'') + \rho(b' \oplus b'') = \rho(a')\rho(a'') + \rho(b')\rho(b''),$ the problem reduces to calculating $\rho(a'),\rho(a''),\rho(b')$ and $\rho(b''),$ each of which is simpler to calculate than $\rho(m)$.

To see why, recall that the $R$-indecomposables are associated with diagonals of two different types: {\textit{dissecting}} and {\textit{clique}}. Draw the diagonal associated with $m$ in the same diagram as the dissecting diagonals of $R$. Then notice that, if we travel from left to right along the diagonal of $m$, two things can happen. We either stay inside a single cell or cross a dissecting diagonal.

If we stay inside a single cell and the cell is empty, then $Gm=0$, so $\rho(m)$ will equal $1$. If the cell is a clique, then the calculation of $\rho(m)$ is more involved and the details are found in the next section.

If $m$ is not inside a cell, then travelling along the diagonal of $m$ inevitably crosses a dissecting diagonal. We show an example. 

The dissecting diagonals are solid, $m$ is $\{3,7\}$, and moving along $m$ first encounters $r = \{2,4\}.$ We then utilise Remark \ref{rem:Ptolemysetup} and fill in the diagonals $a', a'', b', b''$ as in the following figure. \\
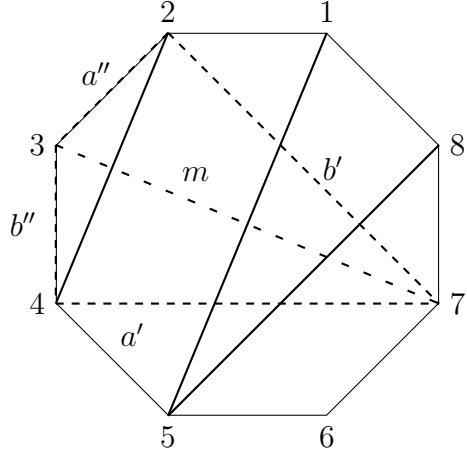
\[
  \begin{tikzpicture}[auto]
    \node[name=s, shape=regular polygon, regular polygon sides=8, minimum size=5.5cm, draw] {}; 
    \draw[thick] (s.corner 2) to (s.corner 4);
    \draw[thick] (s.corner 1) to (s.corner 5);
    \draw[thick] (s.corner 5) to (s.corner 8);
    \draw[thick,loosely dashed] (s.corner 3) to node[pos=0.3] {$m$} (s.corner 7);
    \draw[shift=(s.corner 1)] node[above] {$1$};
    \draw[shift=(s.corner 2)] node[above] {$2$};
    \draw[shift=(s.corner 3)] node[left] {$3$};
    \draw[shift=(s.corner 4)] node[left] {$4$};
    \draw[shift=(s.corner 5)] node[below] {$5$};
    \draw[shift=(s.corner 6)] node[below] {$6$};
    \draw[shift=(s.corner 7)] node[right] {$7$};
    \draw[shift=(s.corner 8)] node[right] {$8$};
    \draw[thick,dashed] (s.corner 2) to node[pos=0.64,above=3pt] {$a''$} (s.corner 3);
   \draw[thick,dashed] (s.corner 3) to node[pos=0.5,left=3pt] {$b''$} (s.corner 4);
   \draw[thick,dashed] (s.corner 4) to node[pos=0.2,below=3pt] {$a'$} (s.corner 7);
   \draw[thick,dashed] (s.corner 2) to node[pos=0.5,right=3pt] {$b'$} (s.corner 7);
  \end{tikzpicture} 
\] \\
We compute $\rho(m)$ as follows:
$$\rho(m) = \rho(a' \oplus a'') + \rho(b' \oplus b'') = \rho(a')\rho(a'') + \rho(b')\rho(b'')$$
$$\Rightarrow \rho(m) = \rho(a') + \rho(b')$$
where the implication on the second line is because $\rho(0) = 1$.

Now notice that the calculation for $\rho(a')$ and $\rho(b')$ will involve only two dissecting diagonals (instead of the original three for $m$). Indeed, it will always be the case that, as the inductive process unfolds, there are fewer and fewer dissecting diagonals involved. This is because (after fixing an orientation) $a''$ and $b''$ are always to the left of the first dissecting diagonal that $m$ crosses, and so will not cross any dissecting diagonals themselves. Furthermore, each of $a'$ and $b'$ will cross fewer dissecting diagonals than $m$ because they lie to the right of the first dissecting diagonal that $m$ crosses. Indeed, in the example above, whereas $m$ crosses the diagonals $\{2,4\},\{1,5\}$ and $\{5,8\}$, $a'$ and $b'$ both only cross $\{1,5\}$ and $\{5,8\}.$

We can then repeat the procedure on each of $a'$ and $b'$ (thereby reducing the number of diagonals again) and continue in this manner until we are finished. As there is a finite number of diagonals in a polygon, the algorithm will finish.

The matter is complicated when there are clique diagonals. When these are present, the same method as above is applied, but the computation of $\rho$ on clique diagonals requires more work, detailed in the forthcoming section. \demo
\end{Remark}

\section{Computing $\rho$ on Clique Diagonals}
\label{sec:CliqDiags}
Recall the setup introduced at the beginning of section \ref{sec:PtolDiags} and let the object $R$ of $\cC$ be given by a Ptolemy diagram $\fP$ in the polygon $P$. Throughout this section, let $\cR = \add(R).$ By the inductive procedure of Section \ref{sec:PtolDiags}, in order to get an algorithm for computing $\rho$, we need to determine $\rho$ on clique diagonals.

A clique $S$ in the Ptolemy diagram corresponding to $R$ may be given by the set of endpoints of its diagonals. This set is a subset of $\{1,2,\dots,n+2,n+3\}$. Consider the example below.

\begin{Example} \label{exa:Ptolemyexample}
Consider the following Ptolemy diagram $\fP$. \\
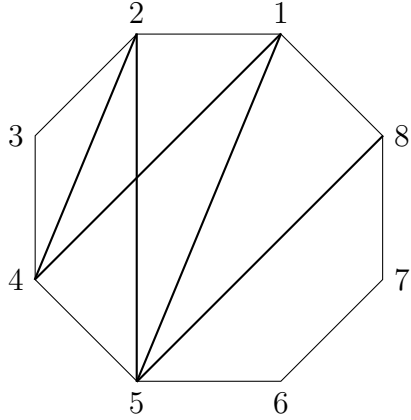
\[
  \begin{tikzpicture}[auto]
    \node[name=s, shape=regular polygon, regular polygon sides=8, minimum size=5cm, draw] {}; 
    \draw[thick] (s.corner 2) to (s.corner 4);
    \draw[thick] (s.corner 1) to (s.corner 5);
    \draw[thick] (s.corner 5) to (s.corner 8);
    \draw[thick] (s.corner 2) to (s.corner 5);
    \draw[thick] (s.corner 1) to (s.corner 4);
        \draw[shift=(s.corner 1)] node[above] {$1$};
    \draw[shift=(s.corner 2)] node[above] {$2$};
    \draw[shift=(s.corner 3)] node[left] {$3$};
    \draw[shift=(s.corner 4)] node[left] {$4$};
    \draw[shift=(s.corner 5)] node[below] {$5$};
    \draw[shift=(s.corner 6)] node[below] {$6$};
    \draw[shift=(s.corner 7)] node[right] {$7$};
    \draw[shift=(s.corner 8)] node[right] {$8$};
  \end{tikzpicture} 
\] \\
The corresponding set $D'$ is $\{\{2,4\},\{1,5\},\{5,8\}\}$; these are the dissecting diagonals of $\fP$. There is one clique, given by the subset $\{1,2,4,5\}$ of $\{1,2,\dots,7,8\}$. The diagonals of the clique are $S=\{\{1,4\},\{2,5\}\}$. Note $\{1,2\}$ and $\{4,5\}$ are not in the clique, as they are edges, and $\{2,4\}$ and $\{1,5\}$ are not in the clique, because they are dissecting diagonals. \demo
\end{Example}

\begin{Setup}
The diagonals of a clique $S$ correspond to a set of indecomposables in $\cC = \sC(A_n).$ Taking $\add$ of these gives $\cS$, which is itself a subcategory of $\cR$. Henceforth, forgiving our abuse of notation we may refer to $\cS$ as simply being a clique, and as explained before in Example \ref{exa:Ptolemyexample}, this clique is given by certain vertices of the $(n+3)$-gon. This permits us to say things like $\cS$ is a clique but $c \in \ind(\cS)$, i.e. $c$ is an indecomposable object corresponding to one of the diagonals in $S$. \demo
\end{Setup}

\begin{Remark}
Recall that if $\cA$ is a $\BC$-linear category, then $\Mod(\cA)$ denotes the abelian $\BC$-linear category of contravariant $\BC$-linear functors $\cA \rightarrow \Vect_\BC$. \demo
\end{Remark}

Recall that $\cR = \add(R)$ and $B=\End_\cC(R).$

\begin{Remark}
By \citep[Prop. 2.7(c)]{ARTAA1}, the functor
$$\Psi : \Mod(\cR) \rightarrow \Mod(B)$$ 
$$F \mapsto F(R)$$
is an equivalence of categories. \demo
\end{Remark}

\begin{Remark}
Consider the functors
$$\Gamma : \cC \rightarrow \Mod(\cR)$$
$$c \mapsto \Hom_\cC(-,\Sigma c)|_{_{\cR}}$$
and
$$G : \cC \rightarrow \Mod(B)$$
$$c \mapsto \Hom_\cC(R,\Sigma c).$$
Note that the following diagram

$$\xymatrix{
\cC \ar[r]^-{\Gamma} \ar[dr]_-{G} & \Mod(\cR) \ar[d]^-{\Psi} \\
& \Mod(B)
}
$$

commutes and due to the fact $\Psi$ is an equivalence, $\rho(c)=\chi(\Gr\,Gc)=\chi(\Gr\,\Gamma c)$ where $\Gr\,\Gamma c = \{M' \subseteq \Gamma c \hspace{1.5mm} | \hspace{1.5mm} M' \textrm{ is a subfunctor}\}$. \demo
\end{Remark}

Henceforth, we will use the notation $\cC(-,-)$ to denote $\Hom_\cC(-,-).$

\begin{Definition}
Define the {\textit{support}} of an additive functor $F : \sC \rightarrow \sD$, denoted $\textsf{Supp}\;F$, to be the set of objects of $\sC$ which under $F$ do not map to the zero object of $\sD$. \demo
\end{Definition}

\begin{Proposition} \label{pro:detbysup}
If $c \in \ind(\cC)$ then a subfunctor $M' \subseteq \Gamma c = \cC(-,\Sigma c)|_{_\cR}$ is determined by $\textsf{Supp}\;M'$.
\end{Proposition}

\begin{proof}
A subfunctor $M' \subseteq \cC(-,\Sigma c)|_{_{\cR}}$ is defined objectwise on indecomposables and inherits morphisms from $\cC(-,\Sigma c)|_{_\cR}$. On indecomposables, this functor is either $0$ or $\BC$ and hence $M'$ is completely determined by its support.
\end{proof}

The Grassmannian $\Gr\,\Gamma c$ consists of different subfunctors $M'$ of $\Gamma c$. If $c \in \ind(\cC)$ then these subfunctors have different supports by Proposition \ref{pro:detbysup}, so each is an isolated point in the Grassmannian. This gives us the following proposition.

\begin{Proposition} \label{pro:rhoccountssubfuncts}
If $c \in \ind(\cC)$, then each $M' \subseteq \Gamma c$ is an isolated point in $\Gr\,\Gamma c$. Therefore $\rho(c)$ is equal to the number of subfunctors of $\Gamma c$.
\end{Proposition}
\begin{proof}
The Grassmannian is a finite collection of isolated points, therefore the Euler characteristic $\chi$ counts the number of points. Therefore $\rho$ counts the number of points and hence the number of subfunctors.
\end{proof}

In the rest of this section, $\cS$ is a clique and $c \in \ind(\cS).$ We will compute $\rho(c)$. By Proposition \ref{pro:rhoccountssubfuncts} we must count the number of subfunctors of $\Gamma c$. This is accomplished by Remarks \ref{rem:countremark1}, \ref{rem:countremark2}, \ref{rem:rectanglesub} and \ref{def:slicedef}.

\begin{Remark} \label{rem:countremark1}
Because $\Mod(\cR)$ has enough projectives, each subfunctor $M' \subseteq \Gamma c$ is the image of a morphism $P \rightarrow \Gamma c$ where $P$ is projective. We can assume $P$ is the direct sum of finitely many indecomposable projectives, each of the form $\cC(-,r)|_{_\cR} = \cR(-,r) = P_r$ for $r \in \ind(\cR).$ This means that $M'$ is the image of a morphism $\displaystyle\bigoplus_{i=1}^m P_{r_i} \rightarrow \Gamma c,$ i.e. $M'$ is the sum of images of morphisms of the form $P_r \rightarrow \Gamma c$.

Let $a : \cR(-,r) \rightarrow \cC(-,\Sigma c)|_{_\cR}$ be such a morphism. By Yoneda's Lemma, it corresponds to an element $\alpha \in \cC(r,\Sigma c)$. Note, $a$ is a natural transformation, and on $r'$ it evaluates to $\xymatrix{\cR(r',r) \ar[r]^-{a_{r'}} & \cC(r',\Sigma c)}$ which is the map given by composition (on the left) by $\alpha.$

Now, there are two possibilities for $r$. Either it is not in the clique $\cS$, which means the diagonals of $r$ and $c$ do not cross. This implies $\cC(r,\Sigma c)=0$, so $\alpha = 0,$ and hence $a=0$. The other possibility for $r$ is that it is in the clique $\cS$. In this case, $\Image(a)$ must be determined.

We have that $\cC(-,\Sigma c)|_{_\cR}$ is supported on $\cS$ because the clique diagonal corresponding to $c$ only crosses those diagonals in the Ptolemy diagram $\fP$ which are in the clique $\cS$. So since $\Image (a)$ is a subfunctor of $\cC(-,\Sigma c)|_{_\cR},$ it is determined by its restriction to $\ind(\cS)$. Let us investigate the functor $\Image(a)$ on $s \in \ind(\cS)$. By objectwise computation, 
$$(\Image(a))(s) = \Image(\{a_s : \cC(s,r) \rightarrow \cC(s,\Sigma c)\})$$
\begin{equation} \label{equ:Images}
\iff (\Image(a))(s) = \Big \{ \begin{matrix} \BC, \hspace{1.5 mm} \textrm{if} \hspace{1.5mm} \exists \hspace{1.5 mm} \sigma : s \rightarrow r \hspace{1.5mm} \textrm{s.t.} \hspace{1.5mm} \alpha\sigma \neq 0, \\ 0, \hspace{1.5 mm} \textrm{if no such} \hspace{1.5 mm} \sigma : s \rightarrow r \hspace{1.5 mm} \textrm{exists.} \hspace{3 mm} \end{matrix}
\end{equation} \demo
\end{Remark}

\begin{Remark} \label{rem:countremark2}
Consider again the indecomposable objects $c, r,$ and $s$ from Remark \ref{rem:countremark1}. Let $s$ be the indecomposable associated with the diagonal $\{s_0,s_1\}$, $r$ the indecomposable associated with the diagonal $\{r_0,r_1\}$, and $\Sigma c$ the indecomposable associated with $\{c_0-1,c_1-1\}$. The diagonals are in the $(n+3)$-gon $P.$ Consider the Auslander-Reiten quiver of $\cC = \sC(A_n).$  Illustrated below with solid lines emanating from $\{s_0,s_1\}$ is the rectangle in the Auslander-Reiten quiver spanned from $s$; this is the region to which $s$ has nonzero morphisms. A similar rectangle can be drawn from $r.$ There will be nonzero morphisms $s \rightarrow r \rightarrow \Sigma c$ with nonzero composition if and only if 
\begin{enumerate}
\item $\{r_0,r_1\}$ is in the rectangle spanned from $\{s_0,s_1\},$ and
\item $\{c_0-1,c_1-1\}$ is inside both the rectangle spanned from $\{s_0,s_1\}$ and the rectangle spanned from $\{r_0,r_1\}.$
\end{enumerate}
This is illustrated below.
$$ \centering{ \def\objectstyle{\scriptstyle}
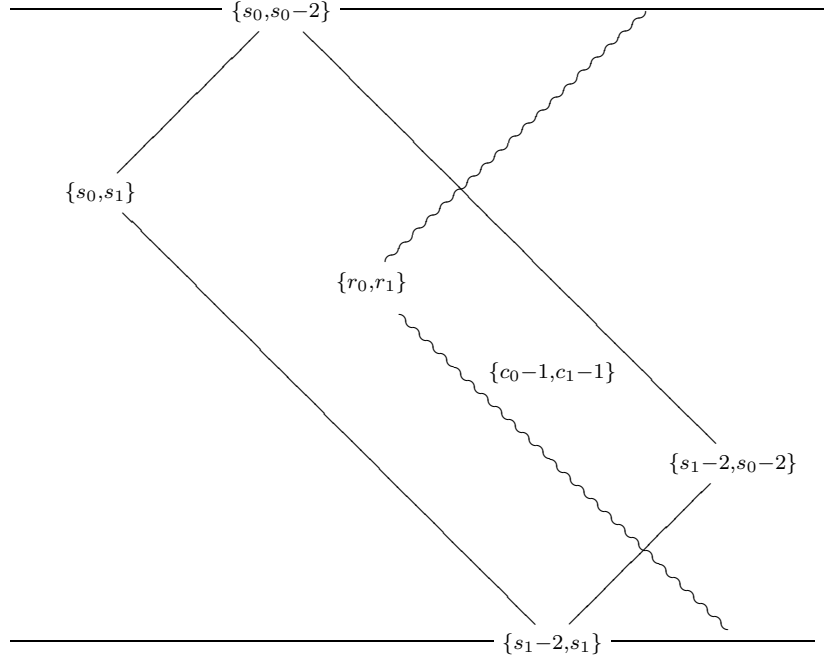
\xymatrix @-3.8pc @!  {
 *{} \ar@{-}[rrr] & & & \{s_0,s_0-2\} \ar@{-}[dddddrrrrr] \ar@{-}[rrrrrr] & & & *{} & *{} &  & *{} \\
 & & & & & & & & \\
 & \{s_0,s_1\} \ar@{-}[uurr] \ar@{-}[dddddrrrrr] & & & & & & & \\
 & & & & \{r_0,r_1\} \ar@{~}[ddddrrrr] \ar@{~}[uuurrr] & & & & \\
 & & & & &  & \{c_0-1,c_1-1\} & & \\
 & & & & & & & & \{s_1-2,s_0-2\} \\
 & & & & & & & & \\
 *{} \ar@{-}[rrrrrr] & & & & & & \{s_1-2,s_1\} \ar@{-}[uurr] \ar@{-}[rrr] & & *{} & \\
} }
$$
In other words, this will happen if and only if
\begin{equation} \label{equ:vertexinequalities}
s_0 \leq r_0 \leq c_0-1 \leq s_1-2 \leq s_1 \leq r_1 \leq c_1-1 \leq s_0-2 
\end{equation}
where the sequence of inequalities means the vertices must occur in this order when moving anti-clockwise around the polygon. We formulate this in a proposition. \demo
\end{Remark}

\begin{Proposition} \label{pro:seqofinequals1}
Let $s \in \ind(\cS)$. There exist $\xymatrix@1{s \ar[r]^-{\sigma} & r \ar[r]^-{\alpha} & \Sigma c}$ with nonzero composition if and only if the endpoints of the corresponding diagonals satisfy the sequence of inequalities (\ref{equ:vertexinequalities}). \demo
\end{Proposition}

\begin{Remark} \label{rem:rectanglesub}
The inequalities in the proposition imply that $c_0+1 \leq r_1 \leq c_1-1$ and $c_1+1 \leq r_0 \leq c_0-1$, hence the diagonals representing $r$ and $c$ cross. The inequalities also imply $c_0+1 \leq s_1 \leq r_1$ and $c_1+1 \leq s_0 \leq r_0$.

We now apply this to our clique $\cS$ and object $c \in \ind(\cS)$.

Consider the rectangle of grid points $X(c)$ below (where the edges have been filled in for illustrative purposes).
$$ \centering{ \def\objectstyle{\scriptstyle}
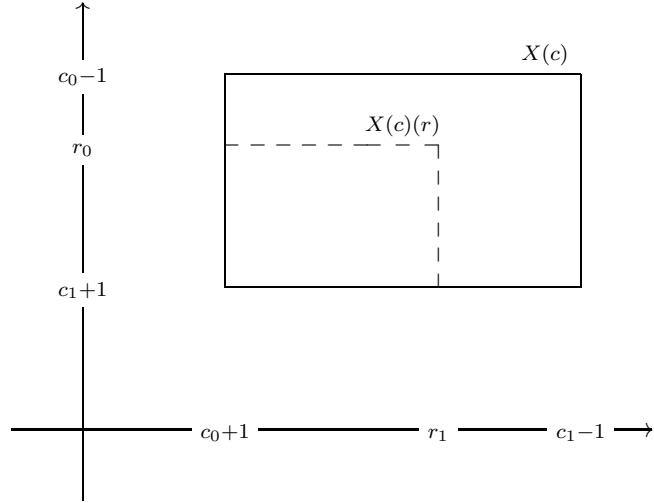
\xymatrix @-1.9pc @!  {
*{} & *{} & *{} & *{} & *{} & *{} & *{} & *{} & *{} & *{}  \\
*{} & c_0-1 \ar[u]  & *{} & *{} \ar@{-}[rrrr] & *{} & *{} & *{} & *{} \ar@{-}[r]^-{X(c)} & *{} & *{}  \\
*{} & r_0 \ar@{-}[u] & *{} & *{} \ar@{--}[rr] & *{} & *{} \ar@{--}[r]^-{X(c)(r)} & *{} & *{} & *{} & *{}  \\
*{} & *{} & *{} & *{} & *{} & *{} & *{} & *{} & *{} & *{}  \\
*{} & c_1+1 \ar@{-}[uu] & *{} & *{} \ar@{-}[uuu] \ar@{-}[rrrrr] & *{} & *{} & *{} \ar@{--}[uu] & *{} & *{} \ar@{-}[uuu] & *{}  \\
*{} & *{} & *{} & *{} & *{} & *{} & *{} & *{} & *{} & *{}  \\
*{} \ar@{-}[rrr] & *{} & *{} & c_0+1 \ar@{-}[rrr] & *{} & *{} & r_1 \ar@{-}[rr] & *{} & c_1-1 \ar[r] & *{} \\
*{} & *{} \ar@{-}[uuu] & *{} & *{} & *{} & *{} & *{} & *{} & *{} & *{} 
}}
$$
On axes, put clique vertices in the indicated intervals. The corresponding set of grid points will be denoted $X(c)$. It corresponds to the diagonals in the clique $\cS$ which cross $\{c_0,c_1\},$ hence to the support of $\Gamma c$ (inside $\ind(\cR)$). The subrectangle of grid points $X(c)(r)$ shows the support of $\Image(a)$, by a combination of the formula (\ref{equ:Images}) and Proposition \ref{pro:seqofinequals1}. \demo
\end{Remark}

\begin{Remark} \label{def:slicedef}
We have now determined the subfunctors $\Image(a) \subseteq \Gamma c.$ Recall Remark \ref{rem:countremark1} which says that the sums of these give all subfunctors of $\Gamma c$. So let us determine these sums.

For each $r \in \ind(\cS)$ we have a subfunctor $E_r = \Image(a) \subseteq \cC(-,\Sigma c)|_{_\cR}$ where $a$ comes from a nonzero map $\alpha : r \rightarrow \Sigma c$. The support of $E_r$ is given by the subrectangle $X(c)(r)$ as indicated above in Remark \ref{rem:rectanglesub}. Now, consider a sum $E_r+E_{r'}$ inside $\cC(-,\Sigma c)|_{_\cR}$ for $r, r' \in \ind(\cS).$ Evaluating at $s \in \textsf{Supp}\;\cC(-,\Sigma c)|_{_\cR}$ gives $E_r(s)+E_{r'}(s) \subseteq \cC(s,\Sigma c)|_{_\cR} = \BC$. This implies that $E_r(s) + E_{r'}(s)$ either equals zero (if both summands are zero) or $\BC$ (if at least one of them is $\BC$). The upshot is that $\textsf{Supp}(E_r+E_{r'})$ is a solid ``staircase'' of grid points (illustrated below).
$$ \centering{ \def\objectstyle{\scriptstyle}
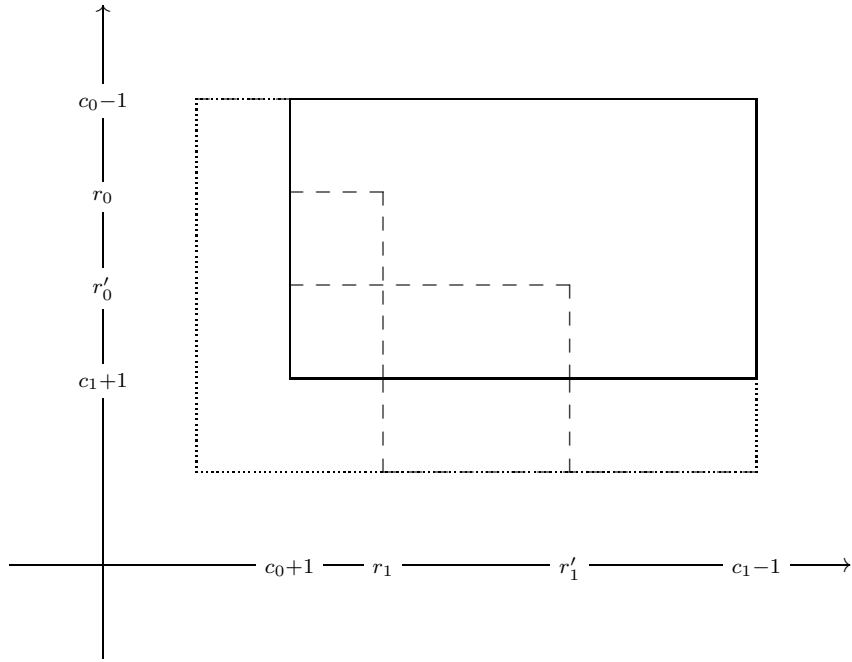
\xymatrix @-1.2pc @!  {
*{} & *{} & *{} & *{} & *{} & *{} & *{} & *{} & *{} & *{}  \\
*{} & c_0-1 \ar[u]  & *{} \ar@{--}[r] \ar@{.}[r]  \ar@{.}[dddd] & *{} \ar@{-}[rrrrr] & *{} & *{} & *{} & *{} & *{} & *{}  \\
*{} & r_0 \ar@{-}[u] & *{} & *{} \ar@{--}[r] & *{} & *{} & *{} & *{} & *{} & *{} \\
*{} & r_0' \ar@{-}[u] & *{} & *{} \ar@{--}[rrr] & *{} & *{} & *{} & *{} & *{} & *{}  \\
*{} & c_1+1 \ar@{-}[u] & *{} & *{} \ar@{-}[uuu] \ar@{-}[rrrrr] & *{} \ar@{--}[uu] \ar@{--}[d] & *{} & *{} \ar@{--}[u] \ar@{--}[d] & *{} & *{} \ar@{-}[uuu] \ar@{.}[d] & *{}  \\
*{} & *{} & *{} \ar@{.}[rrrrrr] & *{} & *{} \ar@{--}[rrrr] & *{} & *{} & *{} & *{} & *{}  \\
*{} \ar@{-}[rrr] & *{} & *{} & c_0+1 \ar@{-}[r] & r_1 \ar@{-}[rr] & *{} & r_1' \ar@{-}[rr] & *{} & c_1-1 \ar[r] & *{} \\
*{} & *{} \ar@{-}[uuu] & *{} & *{} & *{} & *{} & *{} & *{} & *{} & *{}  \\
} }
$$
We want to count the number of such ``staircases''. Instead of counting them directly, it turns out to be easier to count their upper edges. Accordingly, we extend the rectangle of grid points $X(c)$ one vertex to the left and one vertex down, yielding the rectangle $X^{+}(c)$. The upper edge of a staircase should be drawn in $X^{+}(c)$; for instance, the ``empty staircase'' corresponds to an upper edge going straight down, then straight right. Where there are $a$ grid points along the horizontal axis of $X(c)$ and $b$ grid points along the vertical axis of $X(c)$, the corresponding rectangle $X^{+}(c)$ {\textit{in which to embed the upper edge of a staircase}} has side lengths $a$ and $b$. \demo
\end{Remark}

\begin{Theorem}
A subfunctor $M' \subseteq \cC(-,\Sigma c)|_{_\cR}$ has support $\textsf{Supp}\;M'$ equal to a solid staircase of grid points inside the rectangle $X(c)$, which corresponds bijectively to an upper edge of a staircase inside the rectangle $X^{+}(c)$, giving $\textsf{Supp}\;\cC(-,\Sigma c)|_{_\cR}.$ \demo
\end{Theorem}

\begin{Proposition}
There is a bijection between upper edges of staircases in $X^{+}(c)$ and solid staircases of grid points inside $X(c)$, illustrated above, representing the support of a given $\Image(a)$. \demo
\end{Proposition}

\begin{Corollary} \label{cor:submodsequalsstaircases}
The number of submodules of $\cC(-,\Sigma c)|_{_\cR}$ is equal to the number of upper edges of staircases in the rectangle $X^{+}(c)$. \demo
\end{Corollary}

\begin{Lemma}
The number of upper edges of staircases in a rectangle with side lengths $a$ and $b$ is the multimomial coefficient 
$$\displaystyle{{a+b}\choose{a,b}} = \displaystyle\frac{(a+b)!}{a!b!}.$$
\end{Lemma}

\begin{proof}
An upper edge of a staircase is given by a unique permutation of the word $rrr\dots rddd\dots d$ which has $a$ copies of the letter $d$ (``$d$'' for {\textit{down}}) and $b$ copies of the letter $r$ (``$r$'' for {\textit{right}}). This is given by the multinomial coefficient 
$$\displaystyle{{a+b}\choose{a,b}} = \displaystyle\frac{(a+b)!}{a!b!}.$$
\end{proof}
 
\begin{Corollary}
If $c$ is an object corresponding to a clique diagonal, then
$$\rho(c) = \#\textrm{ subfunctors of } \cC(-,\Sigma c)|_{_\cR} = \#\textrm{ upper edges of staircases of } X^{+}(c).$$
\end{Corollary}

\begin{proof}
Proposition \ref{pro:rhoccountssubfuncts} together with the Corollary \ref{cor:submodsequalsstaircases} gives the result.
\end{proof}

Given a clique diagonal $c$, to find the corresponding rectangle $X^{+}(c)$, we may start by drawing the clique, highlighted in an example below.

\begin{Example} \label{Exa:cliquerectangle}
Let us suppose we have a clique inside an octagon and are interested in $c=\{2,7\}$.
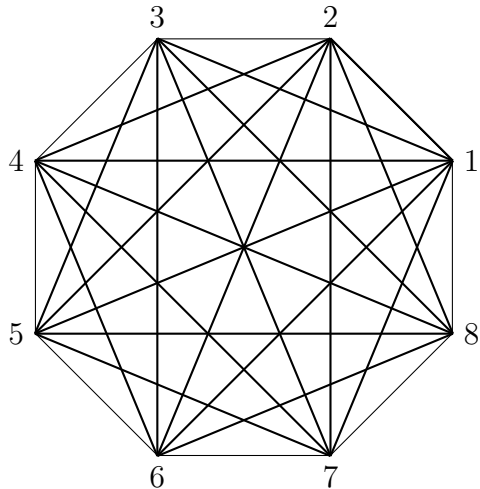
\[
  \begin{tikzpicture}[auto]
    \node[name=s, shape=regular polygon, regular polygon sides=8, minimum size=6cm, draw] {}; 
    \draw[thick] (s.corner 1) to (s.corner 3);
    \draw[thick] (s.corner 1) to (s.corner 4);
    \draw[thick] (s.corner 1) to (s.corner 5);
    \draw[thick] (s.corner 1) to (s.corner 6);
    \draw[thick] (s.corner 1) to (s.corner 7);
    \draw[thick] (s.corner 1) to (s.corner 8);
    \draw[thick] (s.corner 2) to (s.corner 4);
    \draw[thick] (s.corner 2) to (s.corner 5);
    \draw[thick] (s.corner 2) to (s.corner 6);
    \draw[thick] (s.corner 2) to (s.corner 7);
    \draw[thick] (s.corner 2) to (s.corner 8);
    \draw[thick] (s.corner 3) to (s.corner 5);
    \draw[thick] (s.corner 3) to (s.corner 6);
    \draw[thick] (s.corner 3) to (s.corner 7);
    \draw[thick] (s.corner 3) to (s.corner 8);
    \draw[thick] (s.corner 4) to (s.corner 6);
    \draw[thick] (s.corner 4) to (s.corner 7);
    \draw[thick] (s.corner 4) to (s.corner 8);
    \draw[thick] (s.corner 5) to (s.corner 7);
    \draw[thick] (s.corner 5) to (s.corner 8);
    \draw[thick] (s.corner 6) to (s.corner 8);
    \draw[shift=(s.corner 1)] node[above] {$2$};
    \draw[shift=(s.corner 2)] node[above] {$3$};
    \draw[shift=(s.corner 3)] node[left] {$4$};
    \draw[shift=(s.corner 4)] node[left] {$5$};
    \draw[shift=(s.corner 5)] node[below] {$6$};
    \draw[shift=(s.corner 6)] node[below] {$7$};
    \draw[shift=(s.corner 7)] node[right] {$8$};
    \draw[shift=(s.corner 8)] node[right] {$1$};
  \end{tikzpicture} 
\]
The rectangle of grid points $X$ associated with the diagonal $\{2,7\}$ is the following.
$$ \centering{ \def\objectstyle{\scriptstyle}
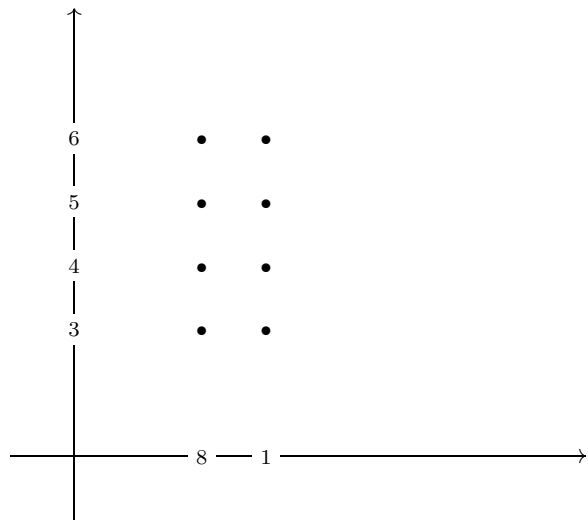
\xymatrix @-1.0pc @!  {
*{} & *{} & *{} & *{} & *{} & *{} & *{} & *{} & *{} & *{} \\
*{} & *{} & *{} & *{} & *{} & *{} & *{} & *{} & *{} & *{} \\
*{} & 6 \ar[uu]  & *{} & \bullet & \bullet & *{} & *{} & *{} & *{} & *{} \\
*{} & 5 \ar@{-}[u] & *{} & \bullet &  \bullet &  *{} & *{} & *{} & *{} & *{} \\
*{} & 4 \ar@{-}[u]& *{} & \bullet &  \bullet &  *{} & *{} & *{} & *{} & *{} \\
*{} & 3 \ar@{-}[u] & *{} & \bullet  & \bullet  & *{} & *{}  & *{} & *{}  & *{} \\
*{} & *{} & *{} & *{} & *{} & *{} & *{} & *{} & *{} & *{} \\
*{} \ar@{-}[rrr] & *{} & *{} & 8 \ar@{-}[r] & 1 \ar[rrrrr] & *{} & *{} & *{} & *{} & *{} \\
*{} & *{} \ar@{-}[uuu] & *{} & *{} & *{} & *{} & *{} & *{} & *{} & *{} \\
} }
$$
To count upper edges of staircases in the rectangle $X^{+}(\{2,7\})$, we extend one vertex in each direction, as follows.
$$ \centering{ \def\objectstyle{\scriptstyle}
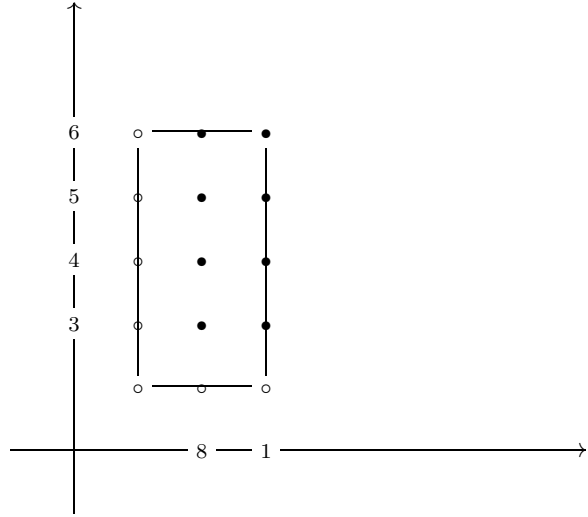
\xymatrix @-1.0pc @!  {
*{} & *{} & *{} & *{} & *{} & *{} & *{} & *{} & *{} & *{} \\
*{} & *{} & *{} & *{} & *{} & *{} & *{} & *{} & *{} & *{} \\
*{} & 6 \ar[uu]  & \circ \ar@{-}[rr] \ar@{-}[dddd] & \bullet &  \bullet & *{} & *{} & *{} & *{} & *{} \\
*{} & 5 \ar@{-}[u] & \circ   & \bullet &  \bullet &  *{} & *{} & *{} & *{} & *{} \\
*{} & 4 \ar@{-}[u]& \circ & \bullet &  \bullet &  *{} & *{} & *{} & *{} & *{} \\
*{} & 3 \ar@{-}[u] & \circ & \bullet  & \bullet  & *{} & *{}  & *{} & *{}  & *{} \\
*{} & *{} & \circ \ar@{-}[rr] & \circ & \circ \ar@{-}[uuuu] & *{} & *{} & *{} & *{} & *{} \\
*{} \ar@{-}[rrr] & *{} & *{} & 8 \ar@{-}[r] & 1 \ar[rrrrr] & *{} & *{} & *{} & *{} & *{} \\
*{} & *{} \ar@{-}[uuu] & *{} & *{} & *{} & *{} & *{} & *{} & *{} & *{}
} }
$$
The number of upper edges of staircases is $\displaystyle\frac{(4+2)!}{4!2!} = 15$. We therefore conclude that $\rho(\{2,7\}) = 15$. \demo
\end{Example}
\begin{Remark}
In general, $\rho$ of a clique diagonal is $\displaystyle\frac{(a+b)!}{a!b!}$ where $a$ and $b$ are the numbers of clique vertices on either side of the clique diagonal.

In an $n$-gon, with vertices numbered $1,\dots,n$, and diagonals ordered $\{a,b\}$ with $a<b,$ these side lengths are always given by $b-a-1$ and $n-b+a-1$. That is, $\rho$ of any diagonal in a clique of size $n$ is given by
$$\displaystyle\frac{(n-2)!}{(b-a-1)!(n-b+a-1)!} = \displaystyle{{n-2}\choose{b-a-1}},$$
a binomial coefficient. \demo 
\end{Remark}
We conclude by calculating $\rho$ of diagonals in an example of a Ptolemy diagram.
\begin{Example}
Consider the following Ptolemy diagram.
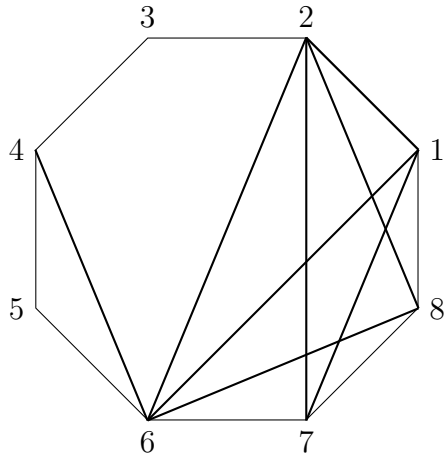
\[
  \begin{tikzpicture}[auto]
    \node[name=s, shape=regular polygon, regular polygon sides=8, minimum size=5.5cm, draw] {}; 
    \draw[thick] (s.corner 1) to (s.corner 5);
    \draw[thick] (s.corner 1) to (s.corner 6);
    \draw[thick] (s.corner 1) to (s.corner 7);
    \draw[thick] (s.corner 1) to (s.corner 8);
    \draw[thick] (s.corner 3) to (s.corner 5);
    \draw[thick] (s.corner 5) to (s.corner 7);
    \draw[thick] (s.corner 5) to (s.corner 8);
    \draw[thick] (s.corner 6) to (s.corner 8);
    \draw[shift=(s.corner 1)] node[above] {$2$};
    \draw[shift=(s.corner 2)] node[above] {$3$};
    \draw[shift=(s.corner 3)] node[left] {$4$};
    \draw[shift=(s.corner 4)] node[left] {$5$};
    \draw[shift=(s.corner 5)] node[below] {$6$};
    \draw[shift=(s.corner 6)] node[below] {$7$};
    \draw[shift=(s.corner 7)] node[right] {$8$};
    \draw[shift=(s.corner 8)] node[right] {$1$};
  \end{tikzpicture} 
\] \\
The diagonal $\{5,8\}$ crosses the dissecting diagonal $\{4,6\}$ so Remark \ref{rem:Ptolemysetup} gives
$$\rho(\{5,8\}) = \rho(\{4,5\})\rho(\{6,8\})+\rho(\{5,6\})\rho(\{4,8\}) = 1 \times \rho(\{6,8\}) + 1 \times \rho(\{4,8\}).$$
In turn, $\{4,8\}$ crosses the dissecting diagonal $\{2,6\}$ so Remark \ref{rem:Ptolemysetup} gives
$$\rho(\{4,8\})=\rho(\{4,6\})\rho(\{2,8\})+\rho(\{2,4\})\rho(\{6,8\})=1 \times \rho(\{2,8\}) + 1 \times \rho(\{6,8\}).$$
Note that $\{4,6\}$ and $\{2,4\}$ cross none of the diagonals in the Ptolemy diagram, so $G$ of these diagonals is $0$ and $\rho$ of them is $1.$ Substituting this into the previous equation gives
$$\rho(\{5,8\}) = \rho(\{2,8\}) + 2 \times \rho(\{6,8\}).$$
Finally, $\{2,8\}$ and $\{6,8\}$ are diagonals in the clique defined by the vertices $\{6,7,8,1,2\}$, and each has one of these vertices on one side, and two on the other, so $\rho(\{2,8\})=\rho(\{6,8\})=\displaystyle\frac{(2+1)!}{2!1!} = 3,$
so $\rho(\{5,8\}) = 3 + 2 \times 3 = 9.$
We can compute $\rho$ values for the remaining diagonals and from this Ptolemy diagram the following values are thus produced, as shown on the Auslander-Reiten quiver of $C(A_5)$.
\[ \def\objectstyle{\scriptstyle}
  \xymatrix @-0.6pc @! {
    \cdots & & 5 \ar@{.}[dd] & & 3 & & 3 & & 3 & & 4 \ar@{.}[dd] & & \cdots \\
    & 4 & & 9 & & 3 & & 3 & & 6 & & 4 & \\
  \cdots  & & 6 \ar@{.}[dd] & & 7 & & 1 & & 4 & & 6 \ar@{.}[dd] & & \cdots \\
    & 6 & & 4 & & 2 & & 1 & & 4 & & 9 & \\
   \cdots & & 4 & & 1 & & 2 & & 1 & & 5 & & \cdots \\
               }
\]
The determinants of four adjacent numbers
\[ 
  \xymatrix @-2.0pc @! {
  & b & \\
  a & & d \\
  & c & \\
               }
\]
forming a diamond are given by $ad-bc$ and in this example, they are either $0,1,$ or $6.$ \demo
\end{Example}

\begin{Remark}
Given a Ptolemy diagram, there is a generalised Caldero-Chapoton map associated with that Ptolemy diagram. When that Ptolemy diagram is a single-cell clique, for different size polygons we arrive at the following values as shown on the relevant Auslander-Reiten quiver (the determinants of each diamond are written in the circles, and boundaries have $\rho$ values of $1$). 
\\ {\textit{(1+3)-gon}}
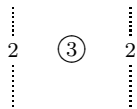
\[ \def\objectstyle{\scriptstyle}
  \xymatrix @-1.2pc @! {
  & & & & \\
  & 2 \ar@{.}[u] \ar@{.}[d] & *+[o][F]{3} & 2 \ar@{.}[u] \ar@{.}[d] & \\
  & & & &
               }
\]
\\ {\textit{(2+3)-gon}}
\[ \def\objectstyle{\scriptstyle}
  
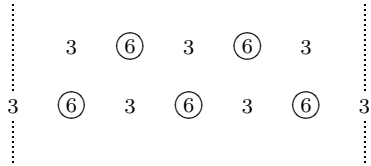
\xymatrix @-1.2pc @! {
  & & & & & & & & \\
  & & 3 & *+[o][F]{6} & 3 & *+[o][F]{6} & 3 & & \\
  & 3 \ar@{.}[uu] \ar@{.}[d] & *+[o][F]{6} & 3 & *+[o][F]{6} & 3 & *+[o][F]{6} & 3 \ar@{.}[uu] \ar@{.}[d] & \\
  & & & & & & & & 
               }
\]
\\ {\textit{(3+3)-gon}}
\[ \def\objectstyle{\scriptstyle}
  
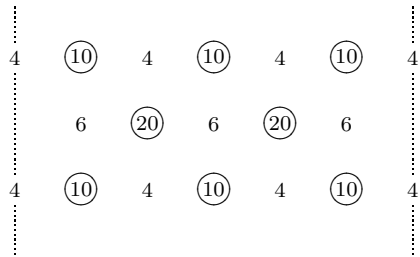
\xymatrix @-1.2pc @! {
  & & & & & & & & \\
 & 4 \ar@{.}[u] & *+[o][F]{10} &4 & *+[o][F]{10} & 4&  *+[o][F]{10} &4 \ar@{.}[u] &  \\
  & & 6& *+[o][F]{20} &6 & *+[o][F]{20} &6 & &  \\
  & 4 \ar@{.}[uu] \ar@{.}[d]& *+[o][F]{10} &4 & *+[o][F]{10} & 4& *+[o][F]{10} &4 \ar@{.}[uu] \ar@{.}[d] &  \\
  & & & & & & & & 
               }
\]
\\ {\textit{(4+3)-gon}}
\[ \def\objectstyle{\scriptstyle}
  
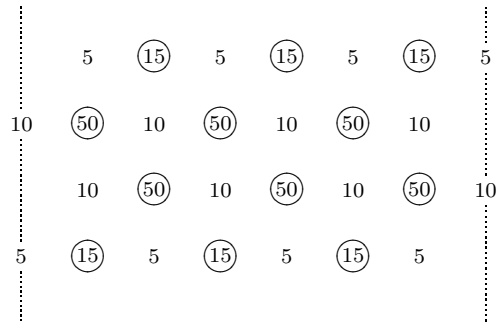
\xymatrix @-1.2pc @! {
  & & & & & & & & &  \\
  & & 5 & *+[o][F]{15} & 5 & *+[o][F]{15} & 5 & *+[o][F]{15} & 5 \ar@{.}[u] &  \\
  & 10 \ar@{.}[uu] & *+[o][F]{50} & 10 & *+[o][F]{50} & 10 & *+[o][F]{50} & 10 & &  \\
  & & 10 & *+[o][F]{50} & 10 & *+[o][F]{50} & 10 & *+[o][F]{50} & 10 \ar@{.}[uu] \ar@{.}[dd] &  \\
  & 5 \ar@{.}[uu] \ar@{.}[d] & *+[o][F]{15} & 5 & *+[o][F]{15} & 5 & *+[o][F]{15} & 5 & &  \\
 & & & & & & & & & & 
               }
\]
As anticipated (see Remark \ref{Exa:cliquerectangle}), binomial coefficients appear. The determinants form part of the sequence known as the ``Triangle of Narayana Numbers'', or the ``Catalan Triangle'', see sequence A001263 in the OEIS \citep{OEIS}. This can be seen by noticing the determinants are of the form
$${{n}\choose{k}}^2 - {{n}\choose{k-1}}{{n}\choose{k+1}}$$
and then utilising the following proposition.
\begin{Proposition}
Suppose $k,n \in \BN$ and $0<k<n$. Then
$$\frac{1}{k+1}{{n}\choose{k}}{{n+1}\choose{k}} = {{n}\choose{k}}^2 - {{n}\choose{k-1}}{{n}\choose{k+1}}.$$
\end{Proposition}
The proof of the proposition is given by straightforward computation. Comparing the expression on the left-hand-side with the formula given in the OEIS gives the result. \demo
\end{Remark}
  \bibliographystyle{is-plain}
  \bibliography{../references}  

\section*{Acknowledgements}
This work was supported in part by the EPSRC Centre for Doctoral Training in Data Science, funded by the UK Engineering and Physical Sciences Research Council (grant reference 1244024). The author is very grateful.

\end{document}